\documentclass[10pt]{article}
\usepackage[T1]{fontenc}
\usepackage{lmodern}
\usepackage{microtype}
\usepackage{amsmath,amssymb,amsthm}
\usepackage[hidelinks]{hyperref}
\usepackage[margin=.9in]{geometry}

\newtheorem{theorem}{Theorem}
\newcommand{\Z}{\mathbb Z}

\newcommand{\vTwo}{\nu_2}

\title{An infinite small-step $\Z^3$-walk with no collinear triple}
\author{Stijn Cambie \and Erik Kalviainen}
\date{September 1, 2026}

\begin{document}
\maketitle

\begin{abstract}
We construct an infinite walk in $\Z^3$ whose steps come from a fixed set of
sixteen vectors and no three of whose vertices are collinear, answering a
problem of Gerver and Ramsey popularized as Erd\H{o}s Problem~193.
\end{abstract}

\begin{theorem}
There is an infinite sequence $P_0,P_1,\ldots$ in $\Z^3$ with no three
collinear terms and with
\[
 P_{n+1}-P_n\in\{-2,-1,0,1,2\}^2\times\{1,2,\ldots,7\}
 \qquad(n\geq0).
\]
Only sixteen successive displacement vectors occur.
\end{theorem}

\begin{proof}
Let $s_2(n)$ be the number of $1$s in the binary expansion of $n$, and put
\[
 u_n=i^{s_2(n)},\qquad z_n=\sum_{0\leq r<n}u_r\in\Z[i].
\]
Thus $(z_n)$ is a nearest-neighbor walk in the Gaussian lattice.  Binary
expansion gives, for $\varepsilon\in\{0,1\}$,
\begin{equation}\label{eq:recurrence}
 u_{2n+\varepsilon}=i^\varepsilon u_n,
 \qquad z_{2n+\varepsilon}=(1+i)z_n+\varepsilon u_n.
\end{equation}

We first record the only property of this walk that we need.  If $0\leq m<n$
and $u_m=u_n$, then
\begin{equation}\label{eq:same-state}
 \vTwo\bigl(|z_n-z_m|^2\bigr)=\vTwo(n-m).
\end{equation}
Indeed, whenever $n-m$ is even, the endpoints have the same parity, say
$m=2a+\varepsilon$ and $n=2b+\varepsilon$.  From
\eqref{eq:recurrence}, $u_m=u_n$ implies $u_a=u_b$ and
\[
 z_n-z_m=(1+i)(z_b-z_a).
\]
If $r=\vTwo(n-m)$, repeating this reduction $r$ times gives
\[
 z_n-z_m=(1+i)^r(z_{n'}-z_{m'}),
\]
where $n'-m'$ is odd.  The last difference is a sum of an odd number of
Gaussian units.  Writing it as $x+iy$, we have $x+y$ odd, so $x^2+y^2$ is
odd.  Since $|1+i|^2=2$, equation~\eqref{eq:same-state} follows.

Let $\alpha_n\in\{0,1,2,3\}$ be determined by $u_n=i^{\alpha_n}$.  Tag the
four states by the corners
\[
 c_0=0,\qquad c_1=i,\qquad c_2=-1+i,\qquad c_3=-1
\]
of a unit square in cyclic order, and define
\begin{equation}\label{eq:points}
 w_n=2z_n+c_{\alpha_n},\qquad
 h_n=4n+\alpha_n,\qquad
 P_n=(\Re w_n,\Im w_n,h_n).
\end{equation}
We claim that, for every $m<n$,
\begin{equation}\label{eq:all-pairs}
 \vTwo\bigl(|w_n-w_m|^2\bigr)=\vTwo(h_n-h_m).
\end{equation}
Put $a=\alpha_m$, $b=\alpha_n$, and $d=n-m$.  Then
\[
 w_n-w_m=2(z_n-z_m)+c_b-c_a,
 \qquad h_n-h_m=4d+b-a.
\]
If $a=b$, equation~\eqref{eq:all-pairs} follows from
\eqref{eq:same-state}.  If $b-a$ is odd, exactly one coordinate of
$c_b-c_a$ is odd; hence the squared modulus and $4d+b-a$ are both odd.  If
$b-a=\pm2$, both planar coordinates are odd, so their squared sum is $2$
modulo $4$, while $4d\pm2$ has valuation one.  This proves
\eqref{eq:all-pairs}.

The sequence has bounded steps.  If $j=\alpha_n$ and $k=\alpha_{n+1}$, then
\begin{equation}\label{eq:steps}
 w_{n+1}-w_n=2i^j+c_k-c_j,
 \qquad h_{n+1}-h_n=4+k-j.
\end{equation}
The second quantity lies between $1$ and $7$.  The corner $c_j$ was oriented
so that $i^j$ points outward from the square.  The component of $c_k-c_j$ in
that direction is $0$ or $-1$, and its perpendicular component has absolute
value at most $1$.  Thus both planar step coordinates have absolute value at
most $2$.  Equation~\eqref{eq:steps} also shows that the step is determined by
$(j,k)$, so at most sixteen steps occur.  Since $h_{n+1}>h_n$, the points are
distinct.

It remains to exclude collinearity.  Suppose $P_a,P_b,P_c$ are collinear for
$a<b<c$, and write
\[
 A=h_b-h_a,\qquad B=h_c-h_b,
 \qquad X=w_b-w_a,\qquad Y=w_c-w_b.
\]
Because $A,B>0$, collinearity gives one common complex slope,
\[
 \frac XA=\frac YB=\frac{X+Y}{A+B}.
\]
Extend $\vTwo$ to nonzero rational numbers in the usual way.  Applying
\eqref{eq:all-pairs} to the three chords and taking squared moduli of the
slopes yields
\[
 \vTwo(A)=\vTwo(B)=\vTwo(A+B).
\]
This is impossible: if the first two valuations equal $t$, then $A/2^t$ and
$B/2^t$ are odd, so $(A+B)/2^t$ is even.  Hence no three points $P_n$ are
collinear.
\end{proof}

\paragraph{Attribution and AI disclosure.}
Erik Kalviainen developed the first proof, relying on a two-dimensional
discrete Hilbert curve, and formalized it in Lean with exact computational
checks and an interactive visualization.  Stijn Cambie became involved at that
point and proposed two successive simplifications: first, encoding all four
Hilbert terminal states directly to eliminate the selector; second, replacing
the Hilbert machinery by the complex Gaussian-lattice walk used in this proof.
Cambie's simplifications and drafts and Kalviainen's formalization and
presentation were AI-assisted.  Both authors have checked the proof and state
the result above as an unconditional theorem.  Neither AI output nor finite
computation is used as a premise.

{\footnotesize\raggedright
\noindent\textbf{Author information.}
Stijn Cambie, Department of Computer Science, KU Leuven Campus
Kulak-Kortrijk, 8500 Kortrijk, Belgium
(\href{mailto:stijn.cambie@hotmail.com}{\nolinkurl{stijn.cambie@hotmail.com}});
Erik Kalviainen, independent researcher, Waterloo, Ontario, Canada
(\href{mailto:ekalvi@gmail.com}{\nolinkurl{ekalvi@gmail.com}}).
S.C. is supported by FWO grant 1225224N.
Supplementary exposition and visualization: \url{https://erdos-193.q5m.ai}.
Source, Lean formalization, and exact checks:
\url{https://github.com/ekalvi/erdos-193}.\par}

\end{document}